\documentclass{amsart}
\usepackage{hyperref}
\usepackage[english]{babel}
\usepackage[utf8]{inputenc}
\usepackage{cite}

\numberwithin{equation}{section}

\theoremstyle{definition}
\newtheorem{Def}{Definition}[section]
\newtheorem{Rem}{Remark}[section]
\newtheorem{Exax}{Example}[section]	
\newenvironment{Exa}
  {\pushQED{\qed}\Exax}
  {\popQED\endExax}
  
\theoremstyle{plain}
\newtheorem{The}{Theorem}[section]
\newtheorem{Pro}{Proposition}[section]
\newtheorem{Lem}{Lemma}[section]

\DeclareMathOperator{\supp}{supp}

\newcommand{\bb}[1]{\mathbb{#1}}
\newcommand{\norm}[1]{\left\|{#1}\right\|}
\newcommand{\abs}[1]{\left|{#1}\right|}
\newcommand{\scalarp}[1]{\left\langle{#1}\right\rangle}

\newcommand{\Appref}[1]{Appendix \ref{#1}}
\newcommand{\Lemref}[1]{Lemma \ref{#1}}
\newcommand{\Secref}[1]{Section \ref{#1}}
\newcommand{\Defref}[1]{Definition \ref{#1}}
\newcommand{\Theref}[1]{Theorem \ref{#1}}
\newcommand{\Proref}[1]{Proposition \ref{#1}}
\newcommand{\Pageref}[1]{page \pageref{#1}}
\newcommand{\Defenuref}[2]{Definition \ref{#1}\eqref{#2}}
\newcommand{\Theenuref}[2]{Theorem \ref{#1}\eqref{#2}}
\newcommand{\Proenuref}[2]{Proposition \ref{#1}\eqref{#2}}

\begin{document}

\title[A Characterization of Sparse Nonstationary Gabor Expansions]{A Characterization of Sparse Nonstationary Gabor Expansions}

\author{Emil Solsb{\ae}k Ottosen}
\address{Department of Mathematical Sciences, Aalborg University, 9220 Aalborg, Denmark}
\email{emilo@math.aau.dk}

\author{Morten Nielsen}
\address{Department of Mathematical Sciences, Aalborg University, 9220 Aalborg, Denmark}
\email{mnielsen@math.aau.dk}



\thanks{Supported by the Danish Council for Independent Research | Natural
Sciences, grant 12-124675, "Mathematical and Statistical Analysis of
Spatial Data”.}

\begin{abstract}
We investigate the problem of constructing sparse time-frequency representations with flexible frequency resolution, studying the theory of nonstationary Gabor frames in the framework of decomposition spaces. Given a painless nonstationary Gabor frame, we construct a compatible decomposition space and prove that the nonstationary Gabor frame forms a Banach frame for the decomposition space. Furthermore, we show that the decomposition space norm can be completely characterized by a sparseness condition on the frame coefficients and we prove an upper bound on the approximation error occurring when thresholding the frame coefficients for signals belonging to the decomposition space. 
\end{abstract}

\maketitle

\section{Introduction}
Redundant Gabor frames play an essential role in time-frequency analysis as these frames provide expansions with good time-frequency resolution \cite{Christensen2003,Grochenig2000}. Gabor frames are based on translation and modulation of a single window function according to lattice parameters which largely determine the redundancy of the frame. By varying the support of the window function one can change the overall resolution of the frame, but it is in general not possible to change the resolution in specific regions of the time-frequency plane. For signals with varying time-frequency characteristics, a fixed resolution is often undesirable. To overcome this problem, the usage of \emph{multi-window Gabor frames} has been proposed \cite{ZibulskiZeevi1997,Dorfler2011, WolfeGodsillDorfler2001,JailletTorresani2007}. As opposed to standard Gabor frames, multi-window Gabor frames use a whole catalogue of window functions of different shapes and sizes to create adaptive representations. A recent example is the \emph{nonstationary Gabor frames} (NSGF's) which have shown great potential in capturing the essential time-frequency information of music signals \cite{NuHAGnsgf2011, NuHAGBeat2011,NuHAG2013,NuHAGconstantQ2011}. These frames use different window functions along \emph{either} the time- or the frequency axes and guarantee perfect reconstruction and an FFT-based implementation in the \emph{painless case}. Originally, NSGF's were studied by Hern{\'a}ndez, Labate \& Weiss \cite{Hernandez2002} and later by Ron \& Shen \cite{Ron2005} who named them \emph{generalized shift-invariant systems}. We choose to work with the terminology introduced in \cite{NuHAGnsgf2011} as we will only consider frames in the painless case for which several practical implementations have been constructed under the name of NSGF's \cite{NuHAGnsgf2011,NuHAG2013,NuHAGconstantQ2011}. We consider painless NSGF's with flexible frequency resolution, corresponding to a sampling grid in the time-frequency plane which is irregular over frequency but regular over time at each fixed frequency position. This construction is particularly useful in connection with music signals since the NSGF can be set to coincide with the semitones used in Western music. Based on the nature of musical tones \cite{PielemeierWakefieldSimoni1996,Dorfler2001}, we expect music signals to permit sparse expansions relative to the redundant NSGF dictionaries. 

The main contribution of this paper is a theoretical characterization of the signals with sparse expansions relative to the NSGF dictionaries. By a sparse expansion we mean an expansion for which the original signal can be approximated at a certain rate by thresholding the expansion coefficients. To prove such a characterization, we follow the approach in \cite{GrochenigSamarah2000, GribonvalNielsen2004, GribonvalNielsen2001} and search for a smoothness space compatible with the structure of the frame. Classical smoothness spaces such as modulation spaces \cite{Feichtinger2003} or Besov spaces \cite{Triebel2010} cannot be expected to be linked with sparse expansions relative to the NSGF dictionaries since these smoothness spaces are not compatible with the flexible frequency resolution of the NSGF's. Modulation spaces correspond to a uniform partition of the frequency domain while Besov spaces correspond to a dyadic partition. Therefore, we study NSGF's in the framework of decomposition spaces. Decomposition spaces were introduced by Feichtinger \& Gröbner in \cite{FeichtingerGrobner1985}, and further studied by Feichtinger in \cite{Feichtinger1987}, and form a large class of function spaces on $\bb{R}^d$ including smoothness spaces such as modulation spaces, Besov spaces, and the intermediate $\alpha-$modulation spaces as special cases \cite{BorupNielsen2007,BorupNielsen2006,Grobner1992}. We construct the decomposition spaces using structured coverings, as introduced by Borup \& Nielsen in \cite{BorupNielsen2007}, which leads to a partition of the frequency domain obtained by applying invertible affine transformations $\{A_k(\cdot)+c_k\}_{k\in \bb{N}}$ on a fixed set $Q\subset \bb{R}^d$. 

Given a painless NSGF, we provide a method for constructing a compatible structured covering and the associated decomposition space. We then show that the NSGF forms a Banach frame for the decomposition space and prove that signals belong to the decomposition space if and only if they permit sparse frame expansions. Based on the sparse expansions, we prove an upper bound on the approximation error occurring when thresholding the frame coefficients for signals belonging to the decomposition space. All these results are based on the characterization given in \Theref{The:MainResult} which is the main contribution of this article. This theorem yields the existence of constants $0<C_1,C_2<\infty$ such that all signals $f$, belonging to the decomposition space $D(\mathcal{Q},L^p,\ell^q_{\omega^s})$, satisfy
\begin{equation*}
C_1\norm{f}_{D(\mathcal{Q},L^p,\ell^q_{\omega^s})}\leq \norm{\left\{\scalarp{f,h^p_{T,n}}\right\}_{T,n}}_{d(\mathcal{Q},\ell^p,\ell^q_{\omega^s})}\leq C_2\norm{f}_{D(\mathcal{Q},L^p,\ell^q_{\omega^s})},
\end{equation*}
with $\{h^p_{T,n}\}_{T,n}$ denoting $L^p-$normalized elements from the NSGF and $d(\mathcal{Q},\ell^p,\ell^q_{\omega^s})$ an associated sequence space. In this way we completely characterize the decomposition space using the frame coefficients from the NSGF.	

The outline of the article is as follows. In \Secref{Sec:2} we define decomposition spaces based on structured coverings and in \Secref{Sec:3} we define NSGF's in the notation of \cite{NuHAGnsgf2011}. We construct the compatible decomposition space in \Secref{Sec:4} and in \Secref{Sec:5} we prove \Theref{The:MainResult}. In \Secref{Sec:6} we show that the NSGF forms a Banach frame for the compatible decomposition space and in \Secref{Sec:7} we provide the link to nonlinear approximation theory.

Let us now introduce some of the notation used throughout this article. We let $\hat{f}(\xi):=\int_{\bb{R}^d}f(x)e^{-2\pi i x\cdot\xi}dx$ denote the Fourier transform with the usual extension to $L^2(\bb{R}^d)$. By $F\asymp G$ we mean that there exist two constants $0<C_1,C_2<\infty$ such that $C_1F\leq G\leq C_2F$. For two (quasi-)normed vector spaces $X$ and $Y$, $X\hookrightarrow Y$ means that $X\subset Y$ and $\norm{f}_Y\leq C\norm{f}_X$ for some constant $C$ and all $f\in X$. We say that a non-empty open set $\Omega'\subset \bb{R}^d$ is \emph{compactly contained} in an open set $\Omega\subset \bb{R}^d$ if $\overline{\Omega'}\subset \Omega$ and $\overline{\Omega'}$ is compact. We denote the matrix norm $\max\{\abs{a_{ij}}\}$ by $\norm{A}_{\ell^\infty(\bb{R}^{d\times d})}$ and we call $\{\xi_i\}_{i\in \mathcal{I}}\subset \bb{R}^d$ a $\delta-$separated set if $\inf_{j,k\in \mathcal{I}, j\neq k}\|\xi_j-\xi_k\|_2=\delta>0$. Finally, by $I_d$ we denote the identity operator on $\bb{R}^d$ and by $\chi_Q$ we denote the indicator function for a set $Q\subset \bb{R}^d$. 
\section{Decomposition Spaces}\label{Sec:2}
In order to construct decomposition spaces, we first need the notion of a structured covering with an associated bounded admissible partitions of unity (BAPU) as defined in \Secref{Sec:StrucCovBAPU}. A BAPU defines a (flexible) partition of the frequency domain corresponding to the structured covering. We use the notation of \cite{BorupNielsen2007} but with slightly modified definitions for both the structured coverings and the BAPU's.
\subsection{Structured covering and BAPU}\label{Sec:StrucCovBAPU}
For an invertible matrix $A\in GL(\bb{R}^d)$, and a constant $c\in \bb{R}^d$, we define the affine transformation
\begin{equation*}
T\xi:=A\xi+c,\quad \xi \in \bb{R}^d.
\end{equation*} 
For a subset $Q\subset \bb{R}^d$ we let $Q_T:=T(Q)$, and for notational convenience we define $|T|:=|\det(A)|$. Given a family $\mathcal{T}=\{A_k(\cdot)+c_k\}_{k\in \bb{N}}$ of invertible affine transformations on $\bb{R}^d$, and a subset $Q\subset \bb{R}^d$, we set $\mathcal{Q}:=\{Q_T\}_{T\in \mathcal{T}}$ and
\begin{equation}\label{eq:neighbouringindices}
\widetilde{T}:=\left\{T'\in \mathcal{T}~\big|~ Q_{T'}\cap Q_T \neq \emptyset\right\},\quad T\in \mathcal{T}.
\end{equation}
We say that $\mathcal{Q}$ is an \emph{admissible covering} of $\bb{R}^d$ if $\bigcup_{T\in \mathcal{T}}Q_T=\bb{R}^d$ and there exists $n_0\in \bb{N}$ such that $|\widetilde{T}|\leq n_0$ for all $T\in \mathcal{T}$. We note that the (minimal) number $n_0$ is the degree of overlap between the sets constituting the covering. 
\begin{Def}[$\mathcal{Q}-$moderate weight]\label{Def:moderateweight}
Let $\mathcal{Q}:=\{Q_T\}_{T\in \mathcal{T}}$ be an admissible covering. A function $u:\bb{R}^d\rightarrow (0,\infty)$ is called $\mathcal{Q}-$moderate if there exists $C>0$ such that $u(x)\leq Cu(y)$ for all $x,y\in Q_T$ and all $T\in \mathcal{T}$. A $\mathcal{Q}-$moderate weight (derived from $u$) is a sequence $\{\omega_T\}_{T\in \mathcal{T}}:=\{u(\xi_T)\}_{T\in \mathcal{T}}$ with $\xi_T\in Q_T$ for all $T\in \mathcal{T}$.
\end{Def}
For the rest of this article we shall use the explicit choice $u(\xi):=1+\|\xi\|_2$ for the function $u$ in \Defref{Def:moderateweight}. We now define the concept of a structured covering, first considered in \cite{BorupNielsen2007}. To ensure that the resulting decomposition spaces are complete, we consider an extended version of the definition given in \cite{BorupNielsen2007}. 
\begin{Def}[Structured covering]\label{Def:SAC}
Given a family $\mathcal{T}=\{A_k(\cdot)+c_k\}_{k\in \bb{N}}$ of invertible affine transformations on $\bb{R}^d$, suppose there exist two bounded open sets $P\subset Q\subset \bb{R}^d$, with $P$ compactly contained in $Q$, such that
\begin{enumerate}
\item $\left\{P_T\right\}_{T\in \mathcal{T}}$ and $\left\{Q_T\right\}_{T\in \mathcal{T}}$ are admissible coverings.\label{Def:SAC1}

\item There exists $K>0$, such that $\norm{A_{k'}^{-1}A_k}_{\ell^\infty(\bb{R}^{d\times d})}\leq K$ holds whenever $(A_{k'}Q+c_{k'})\cap (A_{k}Q+c_k)\neq \emptyset$.\label{Def:SAC2}

\item There exists $K_*>0$, such that $\norm{A_{k}^{-1}}_{\ell^\infty(\bb{R}^{d\times d})}\leq K_*$ holds for all $k\in \bb{N}$.\label{Def:SAC3}

\item There exists a $\delta-$separated set $\{\xi_T\}_{T\in \mathcal{T}}\subset \bb{R}^d$, with $\xi_T\in Q_T$ for all $T\in \mathcal{T}$, such that $\{\omega_T\}_{T\in \mathcal{T}}:=\{u(\xi_T)\}_{T\in \mathcal{T}}$ is a $\mathcal{Q}-$moderate weight.\label{Def:SAC4}

\item There exists $\gamma>0$, such that $|Q_T|\leq \omega_T^\gamma$ for all $T\in \mathcal{T}$. \label{Def:SAC5}
\end{enumerate}
Then we call $\mathcal{Q}=\{Q_T\}_{T\in \mathcal{T}}$ a \emph{structured  covering}.
\end{Def}
\begin{Rem}
\Defref{Def:SAC}\eqref{Def:SAC3}-\eqref{Def:SAC5} are new additions compared to the definition given in \cite{BorupNielsen2007} and are necessary for proving \Theref{THE:DECOMPOSITIONTHEOREM} \Pageref{THE:DECOMPOSITIONTHEOREM}. We note that \Defenuref{Def:SAC}{Def:SAC2} implies $|Q_{T'}|\asymp |Q_{T}|$ uniformly for all $T\in \mathcal{T}$ and all $T'\in \widetilde{T}$, and \Defenuref{Def:SAC}{Def:SAC3} implies a uniform lower bound on $|Q_T|$. 
\end{Rem}
For a structured covering we have the associated concept of a BAPU, first considered in \cite{FeichtingerGrobner1985,BorupNielsen2007}. With a small modification of the proof of \cite[Proposition 1]{BorupNielsen2007} we have the following result. 
\begin{Pro}\label{PRO:BAPU}
Given a structured covering $\mathcal{Q}=\{Q_T\}_{T\in \mathcal{T}}$, there exists a family of non-negative functions $\{\psi_T\}_{T\in \mathcal{T}}\subset C^\infty_c(\bb{R}^d)$ satisfying 
\begin{enumerate}
\item $\supp(\psi_T)\subset Q_T$ for all $T\in \mathcal{T}$.\label{PRO:BAPU1}
\item $\displaystyle \sum_{T \in \mathcal{T}}\psi_T(\xi)=1$ for all $\xi\in \bb{R}^d$.\label{PRO:BAPU2}
\item $\displaystyle\sup_{T\in \mathcal{T}}\abs{Q_T}^{1/p-1}\norm{\mathcal{F}^{-1}\psi_T}_{L^p}<\infty$ for all $0<p\leq 1$.\label{PRO:BAPU3}
\item For all $\alpha \in \bb{N}_0^d$, there exists $C_\alpha>0$ such that $\abs{\partial^{\alpha}\psi_T(\xi)}\leq C_\alpha \chi_{Q_T}(\xi)$, for all $\xi\in \bb{R}^d$ and all $T\in \mathcal{T}$.\label{PRO:BAPU4}
\end{enumerate}
We say that $\{\psi_T\}_{T\in \mathcal{T}}$ is a BAPU subordinate to $\mathcal{Q}$.
\end{Pro}
\begin{Rem}
\Proenuref{PRO:BAPU}{PRO:BAPU3} is necessary to ensure that the decomposition spaces under consideration will be well-defined for $0<p<1$. This case is of specific interest since it plays an essential role in connection with nonlinear approximation theory (cf. \Secref{Sec:7}).
\end{Rem}
\begin{Rem}
\Proenuref{PRO:BAPU}{PRO:BAPU4} is a new addition compared to \cite[Proposition 1]{BorupNielsen2007} and is necessary for proving \Theref{THE:DECOMPOSITIONTHEOREM} \Pageref{THE:DECOMPOSITIONTHEOREM}. The proof of \Proenuref{PRO:BAPU}{PRO:BAPU4} follows easily from the arguments in the proof of \cite[Proposition 1]{BorupNielsen2007} and \Defenuref{Def:SAC}{Def:SAC3}. Finally, it should be noted that the assumptions in \Defref{Def:SAC}\eqref{Def:SAC4}-\eqref{Def:SAC5} are not necessary for proving \Proref{PRO:BAPU}, however, these assumptions are needed for the proof of \Theref{THE:DECOMPOSITIONTHEOREM}.
\end{Rem}
The proof of \cite[Proposition 1]{BorupNielsen2007} is constructive and provides a method for constructing the associated BAPU. Given a structured covering $\{Q_T\}_{T\in \mathcal{T}}$ (with $P$ being compactly contained in $Q$), the method goes as follows:
\begin{enumerate}
\item Pick a non-negative function $\Phi\in C_c^\infty(\bb{R}^d)$ with $\Phi(\xi)= 1$ for all $\xi\in P$ and $\supp(\Phi)\subset Q$.

\item For all $T\in \mathcal{T}$, define
\begin{equation*}
\psi_T(\xi)=\frac{\Phi(T^{-1}\xi)}{\sum_{T'\in \mathcal{T}}\Phi(T'^{-1}\xi)}.
\end{equation*}

\item Then $\{\psi_T\}_{T\in \mathcal{T}}$ is a BAPU subordinate to $\mathcal{Q}=\{Q_T\}_{T\in \mathcal{T}}$.
\end{enumerate}
In the next section we define decomposition spaces based on structured coverings.
\subsection{Definition of decomposition spaces}
Given a structured covering $\mathcal{Q}=\{Q_T\}_{T\in \mathcal{T}}$ with corresponding $\mathcal{Q}-$moderate weight $\{\omega_T\}_{T\in \mathcal{T}}$ and BAPU $\{\psi_T\}_{T\in \mathcal{T}}$. For $s\in \bb{R}$ and $0<q\leq \infty$, we define the associated weighted sequence space $\ell^q_{\omega^s}(\mathcal{T})$ as the sequences of complex numbers $\{a_T\}_{T\in \mathcal{T}}$ satisfying 
\begin{equation*}
\norm{\{a_T\}_{T\in \mathcal{T}}}_{\ell^q_{\omega^s}}:=\norm{\{\omega_T^sa_T\}_{T\in \mathcal{T}}}_{\ell^q}<\infty.
\end{equation*}
Given $\{a_T\}_{T\in \mathcal{T}}\in \ell^q_{\omega^s}(\mathcal{T})$, we define $\{a_T^+\}_{T\in \mathcal{T}}$ by $a_T^+:=\sum_{T'\in \widetilde{T}}a_{T'}$. Since $\{\omega_T\}_{T\in \mathcal{T}}$ is $\mathcal{Q}-$moderate, $\{a_T\}_{T\in \mathcal{T}}\rightarrow \{a_T^+\}_{T\in \mathcal{T}}$ defines a bounded operator on $\ell^q_{\omega^s}(\mathcal{T})$ \cite[Remark 2.13 and Lemma 3.2]{FeichtingerGrobner1985}. Denoting its operator norm by $C_+$, we have 
\begin{equation}\label{eq:Operatornorm}
\norm{\left\{a_T^+\right\}_{T\in \mathcal{T}}}_{\ell^q_{\omega^s}}\leq C_+ \norm{\left\{a_T\right\}_{T\in \mathcal{T}}}_{\ell^q_{\omega^s}},\quad \forall \left\{a_T\right\}_{T\in \mathcal{T}}\in \ell^q_{\omega^s}(\mathcal{T}).
\end{equation}
We will use \eqref{eq:Operatornorm} several times throughout this article. Using the notation of \cite{BorupNielsen2007} we define the Fourier multiplier $\psi_T(D)$ by
\begin{equation*}
\psi_T(D)f:=\mathcal{F}^{-1}(\psi_T \mathcal{F}f),\quad f\in L^2(\bb{R}^d).
\end{equation*}
Combining \Proenuref{PRO:BAPU}{PRO:BAPU3} with \Lemref{Lem:BOproof2} \Pageref{Lem:BOproof2} and \cite[Lemma 1]{BorupNielsen2007} we can show the existence of a uniform constant $C>0$ such that all band-limited functions $f\in L^p(\bb{R}^d)$ satisfy
\begin{equation*}
\norm{\psi_T(D)f}_{L^p}\leq C\norm{f}_{L^p},
\end{equation*}
for all $T\in \mathcal{T}$ and all $0<p\leq\infty$. That is, $\psi_T(D)$ extends to a bounded operator on the band-limited functions in $L^p(\bb{R}^d)$, uniformly in $T\in \mathcal{T}$. Let us now give the definition of decomposition spaces on the Fourier side.
\begin{Def}[Decomposition space]\label{Def:DS}
Let $\mathcal{Q}=\{Q_T\}_{T\in \mathcal{T}}$ be a structured covering of $\bb{R}^d$ with corresponding $\mathcal{Q}-$moderate weight $\{\omega_T\}_{T\in \mathcal{T}}$ and subordinate BAPU $\{\psi_T\}_{T\in \mathcal{T}}$. For $s\in \bb{R}$ and $0<p,q<\infty$, we define the \emph{decomposition space} $D(\mathcal{Q},L^p,\ell^q_{\omega^s})$ as the set of distributions $f\in \mathcal{S}'(\bb{R}^d)$ satisfying
\begin{equation*}
\norm{f}_{D(\mathcal{Q},L^p,\ell^q_{\omega^s})}:=\norm{\left\{\norm{\psi_T(D)f}_{L^p}\right\}_{T\in \mathcal{T}}}_{\ell^q_{\omega^s}}<\infty.
\end{equation*}	
\end{Def}
\begin{Rem}
According to \cite[Theorem 3.7]{FeichtingerGrobner1985}, two different BAPU's yield the same decomposition space with equivalent norms so $D(\mathcal{Q},L^p,\ell^q_{\omega^s})$ is in fact well defined and independent of the BAPU. Actually, the results in \cite{FeichtingerGrobner1985} show that decomposition spaces are invariant under certain geometric modifications of the covering $\mathcal{Q}$, but we will not go into detail here. 
\end{Rem}
\begin{Rem}
In their most general form, decomposition spaces $D(\mathcal{Q},B,Y)$ are constructed using a \emph{local component} $B$ and a \emph{global component} $Y$ \cite{FeichtingerGrobner1985}. This construction is similar to the construction of Wiener amalgam spaces $W(B,C)$ with local component $B$ and global component $C$\cite{Heil2003, Rauhut2007,Feichtinger1983}. However, Wiener amalgam spaces are based on bounded \emph{uniform} partitions of unity, which corresponds to a uniform upper bound on the size of the members of the covering. We do not find such an assumption natural in relation to NSGF's (cf. \Secref{Sec:3}) and have therefore chosen the more general framework of decomposition spaces.
\end{Rem}
In \Theref{THE:DECOMPOSITIONTHEOREM} below we prove that $D(\mathcal{Q},L^p,\ell^q_{\omega^s})$ is in fact a (quasi-)Banach space. Before presenting this result, let us first consider some examples of familiar decomposition spaces. By standard arguments, one can easily show that $D(\mathcal{Q},L^2,\ell^2)=L^2(\bb{R}^d)$ with equivalent norms for any structured covering $\mathcal{Q}$. The next two examples are not as straightforward and demand some structure on the covering. Recall that $\{\xi_T\}_{T\in \mathcal{T}}$ denotes the $\delta-$separated set from \Defref{Def:SAC}\eqref{Def:SAC4}.
\begin{Exa}[Modulation spaces]\label{Exa:Modulation}
Let $Q\subset \bb{R}^d$ be an open cube with center $0$ and side length $r>1$. Define $\mathcal{T}:=\{T_k\}_{k\in \bb{Z}^d}$, with $T_k\xi:=\xi-k$, and set $\xi_{T_k}:=k$ for all $k\in \bb{Z}^d$. With $\mathcal{Q}:=\{Q_T\}_{T\in \mathcal{T}}$ then $D(\mathcal{Q},L^p,\ell^q_{\omega^s})=M^s_{p,q}(\bb{R}^d)$ for $s\in \bb{R}$ and $0<p,q<\infty$, see \cite[Section 4]{Feichtinger2003} for further details.
\end{Exa}
\begin{Exa}[Besov spaces]\label{Exa:Besov}
Let $E_2:=\{\pm 1, \pm 2\}$, $E_1:=\{\pm 1\}$ and $E:=E_2^d\setminus E_1^d$. For $j\in \bb{N}$ and $k\in E$ define $c_{j,k}:=2^j(v(k_1),\ldots,v(k_d))$, where
\begin{equation*}
v(x):=\text{sgn}(x)\cdot  \left\{
     \begin{array}{lr}
       1/2 &  \text{for }x=\pm 1\\
       3/2 &  \text{for }x=\pm 2
     \end{array}
   \right.
\end{equation*}
Let $Q\subset \bb{R}^d$ be an open cube with center $0$ and side length $r>2$. Define $\mathcal{T}:=\{I,T_{j,k}\}_{j\in \bb{N}, k\in E}$, with $T_{j,k}\xi:=2^j\xi+c_{j,k}$, and set $\xi_{T_{j,k}}:=c_{j,k}$ for all $j\in \bb{N}$ and $k\in E$. With $\mathcal{Q}:=\{Q_T\}_{T\in \mathcal{T}}$ then $D(\mathcal{Q},L^p,\ell^q_{\omega^s})=B^s_{p,q}(\bb{R}^d)$ for $s\in \bb{R}$ and $0<p,q<\infty$, see \cite[Section 2.5.4]{Triebel2010} for further details.
\end{Exa}
Let us now study some important properties of decomposition spaces, in particular completeness. 
\begin{The}\label{THE:DECOMPOSITIONTHEOREM}
Given a structured covering $\mathcal{Q}=\{Q_T\}_{T\in \mathcal{T}}$ with $\mathcal{Q}-$moderate weight $\{\omega_T\}_{T\in \mathcal{T}}$ and subordinate BAPU $\{\psi_T\}_{T\in \mathcal{T}}$. For $s\in \bb{R}$ and $0<p,q<\infty$,
\begin{enumerate}
\item $\mathcal{S}(\bb{R}^d)\hookrightarrow D(\mathcal{Q},L^p,\ell^q_{\omega^s}) \hookrightarrow \mathcal{S}'(\bb{R}^d)$.\label{THE:DECOMPOSITIONTHEOREM1}

\item $D(\mathcal{Q},L^p,\ell^q_{\omega^s})$ is a quasi-Banach space (Banach space if $1\leq p,q<\infty$).\label{THE:DECOMPOSITIONTHEOREM2}

\item $\mathcal{S}(\bb{R}^d)$ is dense in $D(\mathcal{Q},L^p,\ell^q_{\omega^s})$.\label{THE:DECOMPOSITIONTHEOREM3}
\end{enumerate}
\end{The}
\begin{Rem}
As was pointed out in \cite{FuhrVoigtlaender2015}, the definition of decomposition spaces given in \cite{BorupNielsen2007} cannot guarantee completeness in the general case. However in \cite{BorupNielsen2008}, this problem was fixed by imposing certain weight conditions on the structured covering. Our proof of \Theref{THE:DECOMPOSITIONTHEOREM} is based on the approach taken in \cite{BorupNielsen2008}. 
\end{Rem}
In \Appref{App:1} we have provided a sketch of the proof for \Theref{THE:DECOMPOSITIONTHEOREM}. The underlying ideas for the proof are similar to those of \cite[Proposition 5.2]{BorupNielsen2008} and several references are made to results in \cite{BorupNielsen2008}. However, in \cite{BorupNielsen2008} the authors considered only coverings made up from open balls and not all arguments carry over to the general case of an arbitrary structured covering.
\section{Nonstationary Gabor Frames}\label{Sec:3}
In this section we define nonstationary Gabor frames with flexible frequency resolution using the notation of \cite{NuHAGnsgf2011}. Given a set of window functions $\{h_m\}_{m\in \bb{Z}^d}$ in $L^2(\bb{R}^d)$, with corresponding time sampling steps $a_m>0$, for $m,n\in \bb{Z}^d$ we define atoms of the form 
\begin{equation*}
h_{m,n}(x):=h_m(x-na_m),\quad x\in \bb{R}^d.
\end{equation*}
The choice of $\bb{Z}^d$ as index set for $m$ is only a matter of notational convenience; any countable index set would do. If $\sum_{m,n}|\langle f,h_{m,n}\rangle|^2\asymp \|f\|_2^2$ for all $f\in L^2(\bb{R}^d)$, we refer to $\{h_{m,n}\}_{m,n}$ as a \emph{nonstationary Gabor frame} (NSGF). For an NSGF $\{h_{m,n}\}_{m,n}$, the frame operator 
\begin{equation*}
Sf=\sum_{m,n\in \bb{Z}^d}\scalarp{f,h_{m,n}}h_{m,n},\quad f\in L^2(\bb{R}^d),
\end{equation*}
is invertible and we have the expansions
\begin{equation*}
f=\sum_{m,n\in \bb{Z}^d}\scalarp{f,h_{m,n}}\tilde{h}_{m,n},\quad  f\in L^2(\bb{R}^d),
\end{equation*}
with $\{\tilde{h}_{m,n}\}_{m,n}:=\{S^{-1}h_{m,n}\}_{m,n}$ being the canonical dual frame of $\{h_{m,n}\}_{m,n}$. An NSGF with flexible frequency resolution corresponds to a grid in the time-frequency plane which is irregular over frequency but regular over time at each frequency position. This property allows for adaptive time-frequency representations as opposed to standard Gabor frames. According to \cite[Corollary 2]{NuHAGnsgf2011}, we have the following important result for NSGF's with band-limited window functions.  
\begin{The}\label{The:NSGFTheorem}
Let $\{h_m\}_{m\in \bb{Z}^d}\subset L^2(\bb{R}^d)$ with time sampling steps $\{a_m\}_{m\in \bb{Z}^d}$, $a_m>0$ for all $m\in \bb{Z}^d$. Assuming $\supp(\hat{h}_m)\subseteq[0,\frac{1}{a_m}]^d+b_m$, with $b_m\in \bb{R}^d$ for all $m\in \bb{Z}^d$, the frame operator for the system
\begin{equation*}
h_{m,n}(x)=h_m(x-na_m),\quad \forall m,n\in \bb{Z}^d,\quad x\in \bb{R}^d,
\end{equation*}
is given by
\begin{equation*}
Sf(x)=\left(\mathcal{F}^{-1}\left(\sum_{m\in \bb{Z}^d}\frac{1}{a_m^d}\abs{\hat{h}_m}^2\right)\ast f\right)(x),\quad f\in L^2(\bb{R}^d).
\end{equation*}
The system $\{h_{m,n}\}_{m,n\in \bb{Z}^d}$ constitutes a frame for $L^2(\bb{R}^d)$, with frame-bounds\\ $0<A\leq B<\infty$, if and only if 
\begin{equation}\label{eq:ConstructDecom1}
A\leq \sum_{m\in \bb{Z}^d}\frac{1}{a_m^d}\abs{\hat{h}_m(\xi)}^2\leq B,\quad \text{for a.e. } \xi\in \bb{R}^d,
\end{equation}
and the canonical dual frame is then given by
\begin{equation}\label{eq:NSGFCanonicalDualFrame}
\tilde{h}_{m,n}(x)=\mathcal{F}^{-1}\left(\frac{\hat{h}_{m}}{\sum_{l\in \bb{Z}^d}\frac{1}{a_l^d}\abs{\hat{h}_l}^2}\right)(x-na_m),\quad x\in \bb{R}^d.
\end{equation}
\end{The}
\begin{Rem}
We note that the canonical dual frame in \eqref{eq:NSGFCanonicalDualFrame} posses the same structure as the original frame, which is a property not shared by general NSGF's. We also note that the canonical tight frame can be obtained by taking the square root of the denominator in \eqref{eq:NSGFCanonicalDualFrame}.
\end{Rem}
Traditionally, an NSGF satisfying the assumptions of \Theref{The:NSGFTheorem} is called a \emph{painless} NSGF, referring to the fact that the frame operator is simply a multiplication operator (in the frequency domain) and therefore easily invertible. This terminology is adopted from the classical \emph{painless nonorthogonal expansions} \cite{Daubechies86}, which corresponds to the painless case for standard Gabor frames. 

By slight abuse of notation we use the term "painless" to denote the NSGF's satisfying \Defref{Def:PainlessNSGF} below. In order to properly formulate this definition, we first need some preliminary notation. Let $\{h_m\}_{m\in \bb{Z}^d}\subset L^2(\bb{R}^d)$ satisfy the assumptions in \Theref{The:NSGFTheorem}. Given $C_*>0$ we denote by $\{I_m\}_{m\in \bb{Z}^d}$ the open cubes
\begin{equation}\label{eq:NSGFopencubes}
I_m:=\left(-\varepsilon_m,\frac{1}{a_m}+\varepsilon_m\right)^d + b_m,\quad  m\in \bb{Z}^d,
\end{equation}
with $\varepsilon_m:=C_*/a_m$ for all $m\in \bb{Z}^d$. We note that $\supp(\hat{h}_{m,n})\subset I_m$ for all $m,n\in \bb{Z}^d$. For $m\in \bb{Z}^d$ we define
\begin{equation*}
\widetilde{m}:=\left\{m'\in \bb{Z}^d~\big| ~ I_{m'}\cap I_{m}\neq \emptyset\right\},
\end{equation*}
using the notation of \eqref{eq:neighbouringindices}. With this definition, $|\widetilde{m}|$ denotes the number of cubes overlapping with $I_m$. Finally, we recall the choice $u(\xi):=1+\|\xi\|_2$ for the function $u$ in \Defref{Def:moderateweight}.
\begin{Def}[Painless NSGF]\label{Def:PainlessNSGF}
Let $\{h_m\}_{m\in \bb{Z}^d}\subset \mathcal{S}(\bb{R}^d)$ satisfy the assumptions in \Theref{The:NSGFTheorem} and assume that
\begin{enumerate}
\item $\{\hat{h}_m\}_{m\in \bb{Z}^d}\subset C_c^\infty(\bb{R}^d)$ and for $\beta \in \bb{N}_0^d$ there exists $C_\beta>0$, such that\label{Def:PainlessNSGF1}
\begin{equation*}
\sup_{\xi\in \bb{R}^d}\abs{\partial_\xi^\beta\hat{h}_m(\xi)}\leq C_\beta a_m^{d/2+\abs{\beta}},\quad \text{for all }m\in \bb{Z}^d.
\end{equation*} 

\item $\sup_{m\in \bb{Z}^d}a_m:=a<\infty$.\label{Def:PainlessNSGF2}

\item There exists $C_*>0$ and $n_0\in \bb{N}$, such that the open cubes $\{I_m\}_{m\in \bb{Z}^d}$ satisfy $|\widetilde{m}|\leq n_0$ and $a_{m'}\asymp a_m$ uniformly for all $m\in \bb{Z}^d$ and all $m'\in \widetilde{m}$.\label{Def:PainlessNSGF3}

\item The centerpoints $\{b_m\}_{m\in \bb{Z}^d}$ forms a $\delta-$separated set and the sequence $\{\omega_m\}_{m\in \bb{Z}^d}:=\{u(b_m)\}_{m\in \bb{Z}^d}$ constitutes a $\{I_m\}_{m\in \bb{Z}^d}-$moderate weight.\label{Def:PainlessNSGF4}

\item There exists $\gamma>0$ such that $|I_m|\leq \omega_m^\gamma$ for all $m\in \bb{Z}^d$.\label{Def:PainlessNSGF5}
\end{enumerate}
Then we refer to $\{h_{m,n}\}_{m,n\in \bb{Z}^d}$ as a \emph{painless} NSGF.
\end{Def}
\begin{Rem}
\Defenuref{Def:PainlessNSGF}{Def:PainlessNSGF2} implies a uniform lower bound on $|I_m|$ and \Defenuref{Def:PainlessNSGF}{Def:PainlessNSGF4} guarantees a minimum distance between the center of the cubes. Furthermore, \Defenuref{Def:PainlessNSGF}{Def:PainlessNSGF3} implies that each cube $I_m$ has at most $n_0$ overlap with other cubes and that the side-length of $I_m$ is equivalent to the side-length of any overlapping cube. 
\end{Rem}
\begin{Rem}\label{Rem:painlessdiscussion}
The assumptions in \Defref{Def:PainlessNSGF} are natural in relation to decomposition spaces and are easily satisfied. However, the support conditions for $\{\hat{h}_m\}_{m\in \bb{Z}^d}$ given in \Theref{The:NSGFTheorem} are rather restrictive and deserves a discussion. In fact, compact support of the window functions is not a necessary assumption for characterizing modulation spaces \cite{Grochenig2000}, Besov spaces \cite{Frazier1985}, or even general decomposition spaces \cite{Nielsen2012}. However, a certain structure of the dual frame is needed and general NSGF's does not provide such structure. We choose to work with the painless case and base our argument on the fact that the dual frame posses the same structure as the original frame. We expect that it is possible to extend the theory developed in this paper to a more general setting by applying existence results for general NSGF's \cite{Dorfler2014,Dofler2015,Holighaus2014} or generalized shift invariant systems \cite{Hernandez2002,Ron2005,Lemvig2016,GittaKutyniok2006}. In particular, the paper \cite{Holighaus2014} by Holighaus seems to provide interesting results in this regard. In this paper, it is shown that for compactly supported window functions, the sampling density in \Theref{The:NSGFTheorem} can (under mild assumptions) be relaxed such that the dual frame posses a structure similar to that of the original frame. However, it is outside the scope of this paper to include such results and we will not go into further details.
\end{Rem}
We now provide a simple example of a set of window functions satisfying \Defenuref{Def:PainlessNSGF}{Def:PainlessNSGF1}.
\begin{Exa}
Let $\varphi\in C_c^\infty(\bb{R}^d)\setminus \{0\}$ with $\supp(\varphi)\subseteq [0,1]^d$ and for $m\in \bb{Z}^d$ define
\begin{equation*}
\hat{h}_m(\xi):=a_m^{d/2}\varphi\left(a_m(\xi-b_m)\right),\quad \forall \xi \in \bb{R}^d,
\end{equation*}
with $b_m\in \bb{R}^d$ and $a_m>0$. Then $\supp(\hat{h}_m)\subseteq[0,\frac{1}{a_m}]^d+b_m$. Furthermore, with $w:=a_m(\xi-b_m)$ the chain rule yields
\begin{equation*}
\abs{\partial_\xi^\beta\hat{h}_m(\xi)}=\abs{\left[\partial^\beta_\xi \varphi\right](w)}a_m^{d/2+\abs{\beta}}\leq C_\beta a_m^{d/2+\abs{\beta}}\chi_{[0,\frac{1}{a_m}]^d+b_m}(\xi),\quad \forall \xi \in \bb{R}^d.
\end{equation*}
This shows \Defenuref{Def:PainlessNSGF}{Def:PainlessNSGF1}.
\end{Exa}
In the next section we consider painless NSGF's in the framework of decomposition spaces in order to characterize signals with sparse expansions relative to the NSGF dictionaries.
\section{Decomposition Spaces Based on Nonstationary Gabor Frames}\label{Sec:4}
We first provide a method for constructing a structured covering which is compatibly with a given painless NSGF $\{h_{m,n}\}_{m,n\in \bb{Z}^d}\subset \mathcal{S}(\bb{R}^d)$. We recall the definition of $\varepsilon_m=C_*/a_m$ used in the construction of $\{I_m\}_{m\in \bb{Z}^d}$ in \eqref{eq:NSGFopencubes}. Define $Q:=(0,1)^d$ together with the set of affine transformations $\mathcal{T}:=\{A_m(\cdot)+c_m\}_{m\in \bb{Z}^d}$ with
\begin{equation*}
A_m:=\left(2\varepsilon_m+\frac{1}{a_m}\right)\cdot I_d\quad \text{and}\quad (c_m)_j:=-\varepsilon_m+(b_m)_j,\quad 1\leq j \leq d.
\end{equation*}
Then $\mathcal{Q}:=\{Q_T\}_{T\in \mathcal{T}}=\{I_m\}_{m\in \bb{Z}^d}$ and, furthermore, we have the following result.
\begin{Lem}
$\mathcal{Q}$ is a structured covering of $\bb{R}^d$.
\end{Lem}
\begin{proof}
Define the set
\begin{equation*}
P:=\left(\frac{C_*}{2C_*+1},\frac{C_*+1}{2C_*+1}\right)^d.
\end{equation*}
By straightforward calculations, it is easy to show that $P$ is compactly contained in $Q$ and $\mathcal{P}:=\{P_T\}_{T\in \mathcal{T}}=\{(0,\frac{1}{a_m})^d+b_m\}_{m\in \bb{Z}^d}$. Let us now show that $\mathcal{P}$ and $\mathcal{Q}$ satisfy the five conditions of \Defref{Def:SAC} \Pageref{Def:SAC}.
\begin{enumerate}
\item First we show that $\mathcal{P}$ covers $\bb{R}^d$. We note that this immediately implies that $\mathcal{Q}$ also covers $\bb{R}^d$. Assume $\mathcal{P}$ does not cover $\bb{R}^d$, i.e. that there exists some $\xi' \in \bb{R}^d$ such that $\xi'\notin (0,\frac{1}{a_m})^d+b_m$ for all $m\in \bb{Z}^d$. Since $\supp(\hat{h}_{m})\subseteq [0,\frac{1}{a_m}]^d+b_m$, and $\hat{h}_{m}$ is continuous, we get $\hat{h}_{m}(\xi')=0$ for all $m\in \bb{Z}^d$. This contradicts the inequality in \eqref{eq:ConstructDecom1} concerning the lower frame bound and thus shows that $\mathcal{P}$ covers $\bb{R}^d$. Now, \Defenuref{Def:PainlessNSGF}{Def:PainlessNSGF3} is precisely the admissibility condition for $\mathcal{Q}$ and thus guarantees that both $\mathcal{P}$ and $\mathcal{Q}$ are admissible coverings. This shows \Defenuref{Def:SAC}{Def:SAC1}.

\item If $(A_{m'}Q+c_{m'})\cap(A_{m}Q+c_{m})\neq \emptyset$, then $a_{m'}\asymp a_m$ according to \Defenuref{Def:PainlessNSGF}{Def:PainlessNSGF3}. Furthermore, since $A_{m'}^{-1}A_m$ is a diagonal matrix and $\varepsilon_m=C_*/a_m$,
\begin{equation*}
\norm{A_{m'}^{-1}A_m}_{\ell^\infty(\bb{R}^{d\times d})}=\frac{a_{m'}}{a_m}\leq \frac{Ka_m}{a_m}=K,
\end{equation*}
for some $K>0$, so \Defenuref{Def:SAC}{Def:SAC2} is satisfied.

\item To show \Defenuref{Def:SAC}{Def:SAC3} we note that 
\begin{equation*}
\norm{A_m^{-1}}_{\ell^\infty(\bb{R}^d\times \bb{R}^d)}=\frac{a_m}{2C_*+1}\leq \frac{a}{2C_*+1},\quad \forall m\in \bb{Z}^d,
\end{equation*}
according to \Defenuref{Def:PainlessNSGF}{Def:PainlessNSGF2}.
\item Finally, \Defref{Def:SAC}\eqref{Def:SAC4}-\eqref{Def:SAC5} follow directly from \Defref{Def:PainlessNSGF}\eqref{Def:PainlessNSGF4}-\eqref{Def:PainlessNSGF5}.
\end{enumerate}
\end{proof}
Since $\mathcal{Q}$ is a structured covering, \Proref{PRO:BAPU} applies and we obtain a BAPU $\{\psi_T\}_{T\in \mathcal{T}}$ subordinate to $\mathcal{Q}$. Given parameters $s\in \bb{R}$ and $0<p,q<\infty$ we may, therefore, construct the associated decomposition space $D(\mathcal{Q},L^p,\ell^q_{\omega^s})$. For notational convenience, we change notation and write $\{h_{T,n}\}_{T\in\mathcal{T},n\in \bb{Z}^d}$, such that $\supp(\hat{h}_{T,n})\subset Q_T$ for all $T\in \mathcal{T}$ and all $n\in \bb{Z}^d$. Since $A_T=(2\varepsilon_T+a_T^{-1})\cdot I_d$, the chain rule and \Defenuref{Def:PainlessNSGF}{Def:PainlessNSGF1} yield
\begin{align}\label{eq:Characterization2}
\abs{\partial_\xi^\beta \left[\hat{h}_T(T\xi)\right]}&=\abs{\left[\partial_\xi^\beta\hat{h}_T\right]\left(T\xi\right)}\cdot \left(2\varepsilon_T+\frac{1}{a_T}\right)^{\abs{\beta}}\notag \\
&\leq C_\beta a_T^{d/2}\cdot \left(2C_*+1\right)^{\abs{\beta}}\chi_{Q_T}(T\xi)=C_\beta'a_T^{d/2}\chi_Q(\xi),\quad \forall \xi \in \bb{R}^d.
\end{align} 
Using \eqref{eq:Characterization2} we can prove the following decay property of $\{h_{T,n}\}_{T,n}$. 
\begin{Pro}\label{Pro:Decay}
For every $N\in \bb{N}$ there exists a constant $C_N>0$ such that for $T=A_T(\cdot)+c_T\in \mathcal{T}$ and $n\in \bb{Z}^d$, 
\begin{equation*}
\abs{h_{T,n}(x)}\leq C_N\abs{T}^{1/2}\left(1+\norm{A_T(x-na_T)}_2\right)^{-N},\quad \forall x\in \bb{R}^d.
\end{equation*}
\end{Pro}
\begin{proof}
We will use the fact that
\begin{equation}\label{eq:Wellknown1}
u(\xi)^N=\left(1+\norm{\xi}_2\right)^N\asymp\sum_{\abs{\beta}\leq N}\abs{\xi^\beta},\quad \xi\in \bb{R}^d,
\end{equation}
for any $N\in \bb{N}$ with $\beta\in \bb{N}^d_0$. Let $\hat{g}_T(\xi):=\hat{h}_T(T\xi)$ such that $\supp(\hat{g}_T)\subset Q$ for all $T\in \mathcal{T}$. Using \eqref{eq:Wellknown1} we get
\begin{align*}
\abs{g_T(x)}&\leq C_1(1+\norm{x}_2)^{-N}\sum_{\abs{\beta}\leq N}\abs{x^\beta g_{T}(x)}\\
&=C_1(1+\norm{x}_2)^{-N}\sum_{\abs{\beta}\leq N}\abs{\mathcal{F}^{-1}\left[\partial_\xi^\beta \hat{g}_{T}\right](x)}\\
&\leq C_1(1+\norm{x}_2)^{-N}\sum_{\abs{\beta}\leq N}\int_{\bb{R}^d}\abs{\partial_\xi^\beta \hat{g}_{T}(\xi)}d\xi,\quad x\in \bb{R}^d.
\end{align*}
Applying \eqref{eq:Characterization2} we may continue and write
\begin{equation}\label{eq:Decay1}
\abs{g_T(x)}\leq C_2a_T^{d/2}(1+\norm{x}_2)^{-N}\sum_{\abs{\beta}\leq N}\int_{\bb{R}^d}\chi_Q(\xi)d\xi=C_3a_T^{d/2}(1+\norm{x}_2)^{-N}.
\end{equation}
Now, since $\varepsilon_T=C_*/a_T$,
\begin{equation}\label{eq:Decay2}
\abs{T}\abs{Q}=\abs{Q_T}=\left(2\varepsilon_T+\frac{1}{a_T}\right)^d=(2C_*+1)^d(a_T)^{-d}.
\end{equation}
Hence, $a_T^{d/2}=C|T|^{-1/2}$ so \eqref{eq:Decay1} yields
\begin{equation}\label{eq:Decay3}
\abs{g_T(x)}\leq C_4\abs{T}^{-1/2}(1+\norm{x}_2)^{-N},\quad x\in \bb{R}^d.
\end{equation}
Using the fact that $A_T$ is a diagonal matrix, we obtain the relationship 
\begin{align}\label{eq:Decay4}
h_T(x)&=\int_{\bb{R}^d}\hat{h}_T(\xi)e^{2\pi i \xi\cdot x}d\xi=\abs{T}\int_{\bb{R}^d}\hat{g}_T(u)e^{2\pi i (A_Tu+c)\cdot x}du\notag\\
&=e^{2\pi ic\cdot x}\abs{T}\int_{\bb{R}^d}\hat{g}_T(u)e^{2\pi iu\cdot A_Tx}du=e^{2\pi ic\cdot x}\abs{T}g_T(A_Tx),\quad x\in \bb{R}^d.
\end{align}
Combining \eqref{eq:Decay4} and \eqref{eq:Decay3} we arrive at
\begin{align*}
\abs{h_{T,n}(x)}&=\abs{h_T(x-na_T)}=\abs{T}\abs{g_T(A_T(x-na_T))}\\
&\leq C_4\abs{T}^{1/2}(1+\norm{A_T(x-na_T)}_2)^{-N},\quad x\in \bb{R}^d.
\end{align*}
This proves the proposition.
\end{proof}
As a direct consequence of \Proref{Pro:Decay} we can prove the following lemma. 
\begin{Lem}
For $0<p<\infty$, we have
\begin{align}
\sup_{x\in \bb{R}^d}\left\{\norm{\left\{h_{T,n}(x)\right\}_{n\in \bb{Z}^d}}_{\ell^p}\right\}&\leq C\abs{T}^{1/2},\quad \text{and}\label{eq:sup1}\\
\sup_{n\in \bb{Z}^d}\norm{h_{T,n}}_{L^p}&\leq C'\abs{T}^{1/2-1/p},\label{eq:sup2}
\end{align}
with constants $C,C'>0$ independent of $T\in \mathcal{T}$.
\end{Lem}
\begin{proof}
We will use the fact that
\begin{equation}\label{eq:Wellknown2}
\int_{\bb{R}^d}u(\xi)^{-m}d\xi=\int_{\bb{R}^d}\left(1+\norm{\xi}_2\right)^{-m}d\xi<\infty,
\end{equation}
for any $m>d$. Choosing $N>d/p$ in \Proref{Pro:Decay}, then \eqref{eq:Wellknown2} yields
\begin{align}\label{eq:MainLemmaProof1}
\norm{\left\{h_{T,n}(x)\right\}_{n\in \bb{Z}^d}}_{\ell^p}&\leq C_1\abs{T}^{1/2}\left(\sum_{n\in \bb{Z}^d}(1+\norm{A_T(x-na_T)}_2)^{-Np}\right)^{1/p}\notag \\
&\leq C_2\abs{T}^{1/2}(a_T^d\abs{T})^{-1/p}.
\end{align}
According to \eqref{eq:Decay2}, $a_T^d=C|T|^{-1}$, which inserted into \eqref{eq:MainLemmaProof1} yields \eqref{eq:sup1}. To show \eqref{eq:sup2}, we again let $N>d/p$ in \Proref{Pro:Decay}, so \eqref{eq:Wellknown2} yields
\begin{equation*}
\norm{h_{T,n}}_{L^p}\leq C_1\abs{T}^{1/2}\left(\int_{\bb{R}^d}(1+\norm{A_T(x-na_T)}_2)^{-Np}dx\right)^{1/p}\leq C_2\abs{T}^{1/2-1/p}.
\end{equation*}
This proves \eqref{eq:sup2}.
\end{proof}
In the next section we use the painless NSGF $\{h_{T,n}\}_{T,n}$ to prove a complete characterization of the corresponding decomposition space $D(\mathcal{Q},L^p,\ell^q_{\omega^s})$.
\section{Characterization of Decomposition Spaces}\label{Sec:5}
The main result of this section is the characterization given in \Theref{The:MainResult}. To prove this result, we follow the approach taken in \cite{BorupNielsen2007} where the authors proved a similar result for a certain type of tight frames for $\bb{R}^d$ (see \cite[Proposition 3]{BorupNielsen2007}). Since the frames we consider are not assumed to be tight we need to modify the arguments given in \cite{BorupNielsen2007}. We start with the following observations.
\begin{Lem}\label{Lem:uniformlyboundedfouriermult}
For $0<p< \infty$, the Fourier multiplier
\begin{equation}\label{eq:uniformlyboundedfouriermult1}
\psi_T^h(D)f:=\mathcal{F}^{-1}\left(\psi_T^h\mathcal{F}f\right):=\mathcal{F}^{-1}\Bigg(\frac{\psi_T}{\sum_{l\in \widetilde{T}}\frac{1}{a_l^d}\abs{\hat{h}_l}^2}\mathcal{F}f\Bigg)
\end{equation}
is bounded on the band-limited functions in $L^p(\bb{R}^d)$ uniformly in $T\in \mathcal{T}$. Further,
\begin{equation}\label{eq:uniformlyboundedfouriermult2}
\sup_{x\in \bb{R}^d}\left\{\norm{\left\{\psi_T^h(D)h_{T',n}\right\}_{n\in \bb{Z}^d}}_{\ell^p}\right\}\leq C\abs{T}^{1/2},\quad T\in \mathcal{T},\quad T'\in \widetilde{T},
\end{equation}
with a constant $C>0$ independent of $T\in \mathcal{T}$.
\end{Lem}
\begin{proof}
Let $\psi_T^{h'}(\xi):=\psi_T^h(T(\xi))$. For $N>d/p$, \eqref{eq:Wellknown2} and \eqref{eq:Wellknown1} imply
\begin{align}\label{eq:lemmauniformoperatorbound1}
\norm{\mathcal{F}^{-1}\psi_T^{h'}}_{L^p}&\leq C_1\norm{u(\cdot)^N\mathcal{F}^{-1}\psi_T^{h'}}_{L^\infty}\leq C_2 \sum_{\abs{\beta}\leq N}\norm{(\cdot)^\beta\mathcal{F}^{-1}\psi_T^{h'}}_{L^\infty}\notag\\
&= C_2 \sum_{\abs{\beta}\leq N}\norm{\mathcal{F}^{-1}\left(\partial^\beta\psi_T^{h'}\right)}_{L^\infty}\leq C_2 \sum_{\abs{\beta}\leq N}\norm{\partial^\beta\psi_T^{h'}}_{L^1}.
\end{align}
Since $\varepsilon_T=C_*/a_T$, the chain rule yields
\begin{equation}\label{eq:lemmauniformoperatorbound2}
\partial^\beta\psi_T^{h'}(\xi)=\left(\partial^\beta\psi_T^h\right)(T\xi)\left(2\varepsilon_T+\frac{1}{a_T}\right)^{\abs{\beta}}=Ca_T^{-\abs{\beta}}\left(\partial^\beta\psi_T^h\right)(T\xi).
\end{equation}
For estimating $\partial^\beta\psi_T^h$ we use the quotient rule. Because all derivatives of $\psi_T$ are bounded according to \Proenuref{PRO:BAPU}{PRO:BAPU4}, we need only to consider the derivatives of the denominator of $\psi_T^h$. The sum in the denominator consists of at most $n_0$ terms and for each term in the sum, the chain rule and \Defref{Def:PainlessNSGF}\eqref{Def:PainlessNSGF1}-\eqref{Def:PainlessNSGF2} imply an upper bound of $Ca_T^{|\beta|}$. Therefore, \eqref{eq:lemmauniformoperatorbound2} yields $|\partial^\beta\psi_T^{h'}(\xi)|\leq C'\chi_{Q}(\xi)$, since $\supp(\psi_T^{h'})\subset Q$ for all $T\in \mathcal{T}$. Combing this with \eqref{eq:lemmauniformoperatorbound1} we get $\|\mathcal{F}^{-1}\psi_T^{h'}\|_{L^p}\leq C_3$. It now follows from \Lemref{Lem:BOproof2} \Pageref{Lem:BOproof2} that $f\rightarrow \mathcal{F}^{-1}(\psi_T^{h'}\mathcal{F}f)$ defines a bounded operator on the band-limited functions in $L^p(\bb{R}^d)$ uniformly in $T\in \mathcal{T}$. Finally, applying \cite[Lemma 1]{BorupNielsen2007} we obtain the same statement for $\psi_T^{h}(D)$.

We now prove \eqref{eq:uniformlyboundedfouriermult2}. Repeating the arguments from the proof of \Proref{Pro:Decay} (using \eqref{eq:ConstructDecom1} and \Defenuref{Def:PainlessNSGF}{Def:PainlessNSGF1}) we can prove the same decay property for $\psi_T^h(D)h_{T',n}$. The result therefore follows from the arguments in the proof of \eqref{eq:sup1}.
\end{proof}
The statement in \Theref{The:MainResult} follows directly once we have proven the following technical lemma. We use the notation $\widetilde{\psi}_T:=\sum_{T'\in \widetilde{T}}\psi_{T'}$.
\begin{Lem}\label{LEM:MAINLEMMA}
Given $f\in \mathcal{S}(\bb{R}^d)$ and $0<p<\infty$. For all $T\in \mathcal{T}$,
\begin{align}
\norm{\left\{\scalarp{f,h_{T,n}}\right\}_{n\in \bb{Z}^d}}_{\ell^p}&\leq C\abs{T}^{1/p-1/2}\norm{\widetilde{\psi}_T(D)f}_{L^p},\quad \text{and}\label{eq:MainLemma1}\\
\norm{\psi_T(D)f}_{L^p}&\leq C'\abs{T}^{1/2-1/p}\sum_{T'\in \widetilde{T}}\norm{\left\{\scalarp{f,h_{T',n}}\right\}_{n\in \bb{Z}^d}}_{\ell^p},\label{eq:MainLemma2}
\end{align}
with constants $C_1,C_2>0$ independent of $T\in \mathcal{T}$.
\end{Lem}
\begin{proof}
The proof of \eqref{eq:MainLemma1} follows directly from \eqref{eq:sup1} and the arguments for the first part of the proof for \cite[Lemma 2]{BorupNielsen2007}. To prove \eqref{eq:MainLemma2} we first assume $p\leq 1$ and note
\begin{align}\label{eq:newLemmaProof1}
\norm{\psi_T(D)f}_{L^p}&\leq C_1\sum_{T'\in \widetilde{T}}\sum_{n\in \bb{Z}^d}\abs{\scalarp{f,h_{T',n}}}\norm{\psi_T(D)\tilde{h}_{T',n}}_{L^p} \\
&\leq C_2\sum_{T'\in \widetilde{T}}\left(\sum_{n\in \bb{Z}^d}\abs{\scalarp{f,h_{T',n}}}^p\norm{\psi_T(D)\tilde{h}_{T',n}}_{L^p}^p\right)^{1/p},\notag
\end{align}
with $\{\tilde{h}_{T,n}\}_{T,n}$ being the dual frame given in \eqref{eq:NSGFCanonicalDualFrame} \Pageref{eq:NSGFCanonicalDualFrame}. Applying \eqref{eq:uniformlyboundedfouriermult1} and \eqref{eq:sup2} this proves \eqref{eq:MainLemma2} for the case $p\leq 1$. For $p> 1$, we note that Hölder's inequality (with $p'$ being the conjugate index of $p$) yields
\begin{align*}
&\norm{\sum_{n\in \bb{Z}^d}\scalarp{f,h_{T',n}}\psi_T(D)\tilde{h}_{T',n}}_{L^p}^p\\
&\leq\int_{\bb{R}^d}\sum_{n\in \bb{Z}^d}\abs{\scalarp{f,h_{T',n}}}^p\abs{\psi_T(D)\tilde{h}_{T',n}(x)}\left(\sum_{n'\in \bb{Z}^d}\abs{\psi_T(D)\tilde{h}_{T',n'}(x)}\right)^{p/p'}dx\\
&\leq C_1\abs{T}^{p/2p'-1/2}\sum_{n\in \bb{Z}^d}\abs{\scalarp{f,h_{T',n}}}^p,
\end{align*}
according to \Lemref{Lem:uniformlyboundedfouriermult} and \eqref{eq:sup2}. Taking the $p$'th root on both sides and applying \eqref{eq:newLemmaProof1} finishes the proof of \eqref{eq:MainLemma2} for $p> 1$.
\end{proof}
Using the notation of \cite{BorupNielsen2007} we define $L^p-$normalized atoms $h^p_{T,n}:=\abs{T}^{1/2-1/p}h_{T,n}$, for all $T\in \mathcal{T}$, $n\in \bb{Z}^d$ and $0<p<\infty$. We also define the coefficient space $d(\mathcal{Q},\ell^p,\ell^q_{\omega^s})$ as the set of coefficients $\{c_{T,n}\}_{T\in \mathcal{T},n\in \bb{Z}^d}\subset \bb{C}$ satisfying
\begin{equation*}
\norm{\{c_{T,n}\}_{T\in \mathcal{T},n\in \bb{Z}^d}}_{d(\mathcal{Q},\ell^p,\ell^q_{\omega^s})}:=\norm{\left\{\norm{\left\{c_{T,n}\right\}_{n\in \bb{Z}^d}}_{\ell^p}\right\}_{T\in \mathcal{T}}}_{\ell^q_{\omega^s}}<\infty.
\end{equation*}
Combining \Lemref{LEM:MAINLEMMA} with the fact that $\mathcal{S}(\bb{R}^d)$ is dense in $D(\mathcal{Q},L^p,\ell^q_{\omega^s})$ we obtain a characterization similar to that of \cite[Proposition 3]{BorupNielsen2007}.
\begin{The}\label{The:MainResult}
For $s\in \bb{R}$ and $0<p,q<\infty$ we have the equivalence
\begin{equation*}
\norm{f}_{D(\mathcal{Q},L^p,\ell^q_{\omega^s})}\asymp \norm{\left\{\scalarp{f,h^p_{T,n}}\right\}_{T\in \mathcal{T},n\in \bb{Z}^d}}_{d(\mathcal{Q},\ell^p,\ell^q_{\omega^s})},
\end{equation*}
for all $f\in D(\mathcal{Q},L^p,\ell^q_{\omega^s})$.
\end{The}
\begin{Rem}
The characterization in \Theref{The:MainResult} differs from the one given in \cite[Proposition 3]{BorupNielsen2007} in two ways. In \cite{BorupNielsen2007} the frame elements are obtained directly from the structured covering such that the resulting system forms a tight frame. In our framework we take the "reverse" approach and explicitly state sufficient conditions which guarantee the existence of a compatible decomposition space for a given NSGF (cf. \Defref{Def:PainlessNSGF}). More importantly, we show that the assumption on tightness of the frame can be replaced with the structured expression for the dual frame given in \eqref{eq:NSGFCanonicalDualFrame} \Pageref{eq:NSGFCanonicalDualFrame}. 
\end{Rem}
In the next section we use the characterization given in \Theref{The:MainResult} to prove that $\{h_{T,n}^p\}_{T,n}$ forms a Banach frame for $D(\mathcal{Q},L^p,\ell^q_{\omega^s})$ with respect to $d(\mathcal{Q},\ell^p,\ell^q_{\omega^s})$ for $s\in \bb{R}$ and $0<p,q<\infty$. 
\section{Banach Frames for Decomposition Spaces}\label{Sec:6}
Let us start by giving the general definition of a Banach frame \cite{Grochenig2000,Grochenig1991}. Traditionally, Banach frames are only defined for Banach spaces but we will also use the concept for quasi-Banach spaces.
\begin{Def}[Banach Frame]\label{Def:BanachFrame}
Let $X$ be a (quasi-)Banach space and let $X_d$ be an associated (quasi-)Banach sequence space on $\bb{N}$. A Banach frame for $X$, with respect to $X_d$, is a sequence $\{y_n\}_{n\in \bb{N}}$ in the dual space $X'$, such that
\begin{enumerate}
\item The coefficient operator $C_X:f\rightarrow \{\langle f,y_n\rangle\}_{n\in \bb{N}}$ is bounded from $X$ into $X_d$.\label{Def:BanachFrame1}

\item Norm equivalence:\label{Def:BanachFrame2}
\begin{equation*}
\norm{f}_X\asymp \norm{\{\scalarp{f,y_n}\}_{n\in \bb{N}}}_{X_d},\quad \forall f\in X.
\end{equation*}

\item There exists a bounded operator $R_{X_d}$ from $X_d$ onto $X$, called a reconstruction operator, such that\label{Def:BanachFrame3}
\begin{equation*}
R_{X_d}C_Xf=R_{X_d}\left(\{\scalarp{f,y_n}\}_{n\in \bb{N}}\right)=f, \quad \forall f\in X.
\end{equation*}
\end{enumerate}
\end{Def}
\begin{Rem}
We will actually prove that $\{h_{T,n}^p\}_{T,n}$ forms an atomic decomposition \cite{Casazza1999,FEICHTINGER1989307,Feichtinger1989II} for $D(\mathcal{Q},L^p,\ell^q_{\omega^s})$ as the reconstruction operator takes the form $f=\sum_{T,n}\langle f,h_{T,n}^p\rangle x_{T,n}$ with $\{x_{T,n}\}\subset D(\mathcal{Q},L^p,\ell^q_{\omega^s})$ (see \Theref{The:BanachFrame} below).
\end{Rem}
In order to show that that $\{h_{T,n}^p\}_{T,n}$ forms a Banach frame for $D(\mathcal{Q},L^p,\ell^q_{\omega^s})$, we first note that 
\begin{equation*}
\{h_{T,n}^p\}_{T\in \mathcal{T},n\in \bb{Z}^d}\subset \mathcal{S}(\bb{R}^d) \subset D'(\mathcal{Q},L^p,\ell^q_{\omega^s})
\end{equation*}
as required by \Defref{Def:BanachFrame}. Furthermore, the equivalence in \Theref{The:MainResult} implies that \Defenuref{Def:BanachFrame}{Def:BanachFrame2} is satisfied and the corresponding proof reveals that \Defenuref{Def:BanachFrame}{Def:BanachFrame1} is satisfied. What remains to be shown is the existence of a bounded reconstruction operator such that \Defenuref{Def:BanachFrame}{Def:BanachFrame3} holds. For $\{c_{T,n}\}_{T,n}\in d(\mathcal{Q},\ell^p,\ell^q_{\omega^s})$, we define the reconstruction operator as
\begin{equation}\label{eq:DefReconstructionOperator}
R_{d(\mathcal{Q},\ell^p,\ell^q_{\omega^s})}\left(\left\{c_{T,n}\right\}_{T,n}\right)=\sum_{T\in \mathcal{T},n\in \bb{Z}^d}c_{T,n}\abs{T}^{1/p-1/2}\tilde{h}_{T,n},
\end{equation}
with $\{\tilde{h}_{T,n}\}_{T\in \mathcal{T},n\in \bb{Z}^d}$ being the dual frame given in \eqref{eq:NSGFCanonicalDualFrame} \Pageref{eq:NSGFCanonicalDualFrame}. We now provide the main result of this section.
\begin{The}\label{The:BanachFrame}
Given $s\in \bb{R}$ and $0<p,q<\infty$, $\{h^p_{T,n}\}_{T\in \mathcal{T},n\in \bb{Z}^d}$ forms a Banach frame for $D(\mathcal{Q},L^p,\ell^q_{\omega^s})$. Furthermore, we have the expansions
\begin{equation}\label{eq:AtomicExpansion}
f=\sum_{T\in \mathcal{T},n\in \bb{Z}^d}\scalarp{f,h_{T,n}}\tilde{h}_{T,n},\quad \forall f\in D(\mathcal{Q},L^p,\ell^q_{\omega^s}),
\end{equation}
with unconditional convergence.
\end{The}
\begin{proof}
Let $R$ and $C$ denote the reconstruction- and coefficient operator, respectively. We first prove that $R$ is bounded from $d(\mathcal{Q},\ell^p,\ell^q_{\omega^s})$ onto $D(\mathcal{Q},L^p,\ell^q_{\omega^s})$. For $\{c_{T,n}\}_{T,n}\in d(\mathcal{Q},\ell^p,\ell^q_{\omega^s})$ we let $g:=R(\{c_{T,n}\}_{T,n})$. For $T\in \mathcal{T}$, \Lemref{Lem:uniformlyboundedfouriermult} implies
\begin{align}\label{eq:ReconstructionOperator1}
\norm{\psi_T(D)g}_{L^p}&=\norm{\mathcal{F}^{-1}\left(\frac{\psi_T}{\sum_{l\in \widetilde{T}}\frac{1}{a_l^d}\abs{\hat{h}_l}^2}\cdot\widetilde{\psi}_T\cdot\sum_{T'\in \mathcal{T},n\in \bb{Z}^d}c_{T',n}\abs{T'}^{1/p-1/2}\hat{h}_{T',n}\right)}_{L^p}\notag\\
&\leq C_1\norm{\widetilde{\psi}_T(D)\left(\sum_{T'\in \mathcal{T},n\in \bb{Z}^d}c_{T',n}\abs{T'}^{1/p-1/2}h_{T',n}\right)}_{L^p}.
\end{align}
Repeating the arguments from the proof of \cite[Lemma 4]{BorupNielsen2007} we can show that
\begin{equation}\label{eq:ReconstructionOperator2}
\norm{\sum_{T\in \mathcal{T},n\in \bb{Z}^d}c_{T,n}\abs{T}^{1/p-1/2}h_{T,n}}_{D(\mathcal{Q},L^p,\ell^q_{\omega^s})}\leq C\norm{\{c_{T,n}\}_{T,n}}_{d(\mathcal{Q},\ell^p,\ell^q_{\omega^s})}.
\end{equation}
Applying \eqref{eq:Operatornorm} \Pageref{eq:Operatornorm} to \eqref{eq:ReconstructionOperator1} and then using \eqref{eq:ReconstructionOperator2} we get
\begin{align}\label{eq:ReconstructionOperator3}
\norm{g}_{D(\mathcal{Q},L^p,\ell^q_{\omega^s})}&\leq C_2\norm{\sum_{T\in \mathcal{T},n\in \bb{Z}^d}c_{T,n}\abs{T}^{1/p-1/2}h_{T,n}}_{D(\mathcal{Q},L^p,\ell^q_{\omega^s})}\notag\\
&\leq C_3\norm{\{c_{T,n}\}_{T\in \mathcal{T},n\in \bb{Z}^d}}_{d(\mathcal{Q},\ell^p,\ell^q_{\omega^s})}.
\end{align}
This proves that $R$ is bounded from $d(\mathcal{Q},\ell^p,\ell^q_{\omega^s})$ onto $D(\mathcal{Q},L^p,\ell^q_{\omega^s})$. Let us now show the unconditional convergence of \eqref{eq:AtomicExpansion}.  Given $f\in D(\mathcal{Q},L^p,\ell^q_{\omega^s})$, we can find a sequence $\{f_k\}_{k\geq 1}$, with $f_k\in \mathcal{S}(\bb{R}^d)$ for all $k\geq 1$, such that $f_k\rightarrow f$ in $D(\mathcal{Q},L^p,\ell^q_{\omega^s})$ as $k\rightarrow \infty$. Furthermore, since $\{h_{T,n}\}_{T,n}$ forms a frame for $L^2(\bb{R}^d)$, for each $k\geq 1$ we have the expansion
\begin{equation*}
f_k=\sum_{T\in \mathcal{T},n\in \bb{Z}^d}\scalarp{f_k,h_{T,n}}\tilde{h}_{T,n}=RC(f_k),
\end{equation*}
with unconditional convergence. Since $RC:D(\mathcal{Q},L^p,\ell^q_{\omega^s})\rightarrow D(\mathcal{Q},L^p,\ell^q_{\omega^s})$ is continuous, letting $k\rightarrow \infty$ yields
\begin{equation}\label{eq:BanachFrameProof1}
f=RC(f)=\sum_{T\in \mathcal{T},n\in \bb{Z}^d}\scalarp{f,h_{T,n}}\tilde{h}_{T,n}.
\end{equation}
Given $\varepsilon>0$, \eqref{eq:ReconstructionOperator3} implies that we can find a \emph{finite} subset $F_0\subset \mathcal{T}\times\bb{Z}^d$, such that 
\begin{equation*}
\norm{f-\sum_{(T,n)\in F}\scalarp{f,h_{T,n}}\tilde{h}_{T,n}}_{D(\mathcal{Q},L^p,\ell^q_{\omega^s})}\leq C \norm{\left\{\scalarp{f,h_{T,n}}\right\}_{(T,n)\notin F}}_{d(\mathcal{Q},L^p,\ell^q_{\omega^s})}<\varepsilon,
\end{equation*}
for all finite sets $F\supseteq F_0$. According to \cite[Proposition 5.3.1]{Grochenig2000}, this property is equivalent to unconditional convergence.
\end{proof}
We close this section by discussing the implications of the achieved results. According to \Theref{The:MainResult} and \Theref{The:BanachFrame}, every $f\in D(\mathcal{Q},L^p,\ell^q_{\omega^s})$ has an expansion of the form
\begin{align*}
f&=\sum_{T\in \mathcal{T},n\in \bb{Z}^d}\scalarp{f,h^p_{T,n}}\abs{T}^{1/p-1/2}\tilde{h}_{T,n},\quad\text{with}\\
\norm{f}_{D(\mathcal{Q},L^p,\ell^q_{\omega^s})}&\asymp \norm{\left\{\scalarp{f,h^p_{T,n}}\right\}_{T,n}}_{d(\mathcal{Q},\ell^p,\ell^q_{\omega^s})}.
\end{align*}
Now, assume there exists another set of reconstruction coefficients $\{c_{T,n}\}_{T,n}\in d(\mathcal{Q},\ell^p,\ell^q_{\omega^s})$ which is sparser than $\{\langle f,h^p_{T,n}\rangle\}_{T,n}$ when sparseness is measured by the $d(\mathcal{Q},\ell^p,\ell^q_{\omega^s})$-norm. Since the reconstruction operator $R$ is bounded we get
\begin{align*}
\norm{\{c_{T,n}\}_{T,n}}_{d(\mathcal{Q},\ell^p,\ell^q_{\omega^s})}&\leq \norm{\left\{\scalarp{f,h^p_{T,n}}\right\}_{T,n}}_{d(\mathcal{Q},\ell^p,\ell^q_{\omega^s})}\leq C_1 \norm{f}_{D(\mathcal{Q},L^p,\ell^q_{\omega^s})}\\
&= C_1 \norm{R(\{c_{T,n}\}_{T,n})}_{D(\mathcal{Q},L^p,\ell^q_{\omega^s})}\leq C_2 \norm{\{c_{T,n}\}_{T,n}}_{d(\mathcal{Q},\ell^p,\ell^q_{\omega^s})}.
\end{align*}
We conclude that the canonical coefficients $\{\langle f,h^p_{T,n}\rangle\}_{T,n}$ are (up to a constant) the sparsest possible choice for expanding $f$ as 
\begin{equation*}
f=\sum_{T\in \mathcal{T},n\in \bb{Z}^d}c_{T,n}\abs{T}^{1/p-1/2}\tilde{h}_{T,n},
\end{equation*}
when sparseness of the coefficients is measured by the $d(\mathcal{Q},\ell^p,\ell^q_{\omega^s})$-norm. Furthermore, $f\in D(\mathcal{Q},L^p,\ell^q_{\omega^s})$ if and only if $f$ permits a sparse expansion relative to the dictionary $\{|T|^{1/p-1/2}\widetilde{h}_{T,n}\}_{T,n}$ .
\section{Application to Nonlinear Approximation Theory}\label{Sec:7}
In this section we provide the link to nonlinear approximation theory. An important property of the sparse expansions obtained in \Theref{The:BanachFrame} is that we can obtain a good compression by simply thresholding the coefficients from the expansion. As mentioned in the introduction, NSGF's can create adaptive time-frequency representations as opposed to standard Gabor frames. Such adaptive representations can be constructed to fit the particular nature of a given signal, thereby producing a more precise (and hopefully sparser) time-frequency representation. In particular, NSGF's have proven to be useful in connection with music signals. For instance, in \cite{NuHAGnsgf2011,NuHAGconstantQ2011} the authors use NSGF's to construct an invertible constant-Q transform with good frequency resolution at the lower frequencies and good time resolution at the higher frequencies. Such a time-frequency resolution is often more natural for music signals than the uniform resolution provided by Gabor frames.

The main result of this section is given in \eqref{eq:ApplicationNonlinear2} below. The corresponding proof follows directly from the results obtained in Sections \ref{Sec:5} and \ref{Sec:6} together with standard arguments from nonlinear approximation theory \cite{GribonvalNielsen2004}. Let $f\in D(\mathcal{Q},L^\tau,\ell^\tau_{\omega^s})$, with $0<\tau<\infty$, and let $0<p<\infty$ satisfy $\alpha:=1/\tau-1/p>0$. Write the frame expansion of $f$ with respect to the $L^p-$normalized coefficients
\begin{equation}\label{eq:ApplicationNonlinear1}
f=\sum_{T\in \mathcal{T},n\in \bb{Z}^d}\scalarp{f,h^p_{T,n}}\abs{T}^{1/p-1/2}\tilde{h}_{T,n}.
\end{equation}
Let $\{\theta_m\}_{m\in \bb{N}}$ be a decreasing rearrangement of the frame coefficients and let $f_N$ be the $N$-term approximation to $f$ obtained by extracting the coefficients in \eqref{eq:ApplicationNonlinear1} corresponding to the $N$ largest coefficients $\{\theta_m\}_{m=1}^N$. Then, we can prove the existence of $C>0$ such that for $f\in D(\mathcal{Q},L^\tau,\ell^\tau_{\omega^s})$ and $N\in \bb{N}$, 
\begin{equation}\label{eq:ApplicationNonlinear2}
\norm{f-f_N}_{D(\mathcal{Q},L^p,\ell^p_{\omega^s})}\leq C N^{-\alpha}\norm{f}_{D(\mathcal{Q},L^\tau,\ell^\tau_{\omega^s})}.
\end{equation}
In other words, for $f\in D(\mathcal{Q},L^\tau,\ell^\tau_{\omega^s})$ we can obtain good approximations in $D(\mathcal{Q},L^p,\ell^p_{\omega^s})$ by thresholding the $L^p-$normalized frame coefficients. We note that for $0<\tau<2$ we obtain good approximations in $L^2(\bb{R}^d)$ with respect to the original coefficients $\{\langle f,h_{T,n}\rangle\}_{T,n}$. 

We now explain the obtained results in the general framework of Jackson- and Bernstein inequalities \cite{Devore1993}. Let $\mathcal{D}$ denote the dictionary $\{|T|^{1/p-1/2}\widetilde{h}_{T,n}\}_{T,n}$ and define the nonlinear set of all linear combinations of at most $N$ elements from $\mathcal{D}$ as
\begin{equation*}
\Sigma_N(\mathcal{D}):=\left\{\sum_{T,n\in \Delta}c_{T,n}\abs{T}^{1/p-1/2}\tilde{h}_{T,n}~\Bigg|~\#\Delta\leq N\right\}.
\end{equation*}
For any $f\in D(\mathcal{Q},L^p,\ell^p_{\omega^s})$, the error of best $N$-term approximation to $f$ is 
\begin{equation*}
\sigma_N\left(f,\mathcal{D}\right):=\inf_{h\in \Sigma_N(\mathcal{D})}\norm{f-h}_{D(\mathcal{Q},L^p,\ell^p_{\omega^s})}.
\end{equation*}
Since $f_N\in \Sigma_N(\mathcal{D})$, \eqref{eq:ApplicationNonlinear2} yields 
\begin{equation*}
\sigma_N\left(f,\mathcal{D}\right)\leq CN^{-\alpha}\norm{f}_{D(\mathcal{Q},L^\tau,\ell^\tau_{\omega^s})}.
\end{equation*}
This is a so-called Jackson inequality for nonlinear $N$-term approximation with $\mathcal{D}$. It provides us with an upper bound for the error obtained by approximating $f$ with the best possible choice of linear combinations of at most $N$ elements from the dictionary. The converse inequality is called a Bernstein inequality and is in general much more difficult to obtain for redundant systems \cite{NielsenGribonval2006}. The existence of a Bernstein inequality would provide us with a lower bound and hence a full characterization of the error of best $N$-term approximation to $f$ with respect to the dictionary $\mathcal{D}$. However, for this particular system (and for many other redundant systems), the existence of a Bernstein inequality is still an open question.
\section*{Acknowledgements}
We thank the two anonymous reviewers for their constructive comments on the original manuscript. Their valuable suggestions have helped improve the manuscript considerably.

\appendix
\section{Proof of \Theref{THE:DECOMPOSITIONTHEOREM}}\label{App:1}
\begin{proof}
To simplify notation we let $D^s_{p,q}:=D(\mathcal{Q},L^p,\ell^q_{\omega^s})$. Let us first prove \Theenuref{THE:DECOMPOSITIONTHEOREM}{THE:DECOMPOSITIONTHEOREM1}. Allowing the extension $q=\infty$, and repeating the arguments from the proof of \cite[Proposition 5.7]{BorupNielsen2008}, we can show that
\begin{equation*}
D^{s+\varepsilon}_{p,\infty}\hookrightarrow D^s_{p,q}\hookrightarrow D^s_{p,\infty},\quad \varepsilon>d/q,
\end{equation*}
for any $s\in \bb{R}$ and $0<p<\infty$ using \Defref{Def:SAC}\eqref{Def:SAC4}. It therefore suffice to show that $\mathcal{S}(\bb{R}^d)\hookrightarrow D^s_{p,\infty} \hookrightarrow \mathcal{S}'(\bb{R}^d)$ for any $s\in \bb{R}$ and $0<p<\infty$. For $N\in \bb{N}$, we define semi-norms on $\mathcal{S}(\bb{R}^d)$ by
\begin{equation*}
p_N(g):=\sup_{\xi \in \bb{R}^d}\left\{u(\xi)^N\sum_{\abs{\beta}\leq N}\abs{\partial^\beta \hat{g}(\xi)}\right\},\quad g\in \mathcal{S}(\bb{R}^d),
\end{equation*}
with $u(\xi)=1+\|\xi\|_2$ as usual. Following the approach in \cite[Page 149]{BorupNielsen2008}, and applying \Proenuref{PRO:BAPU}{PRO:BAPU4}, we get
\begin{equation*}
\norm{f}_{D^s_{p,\infty}}\leq Cp_N(f)\quad \text{and}\quad \norm{f}_{D^s_{1,1}}\leq C'p_{N'}(f),
\end{equation*}
for sufficiently large $N$ and $N'$. This proves that $\mathcal{S}(\bb{R}^d)\hookrightarrow D^s_{p,\infty}$ and $\mathcal{S}(\bb{R}^d)\hookrightarrow D^s_{1,1}$. To show that $D^s_{p,\infty} \hookrightarrow \mathcal{S}'(\bb{R}^d)$ we need to take a different approach than in \cite{BorupNielsen2008}. Setting $\widetilde{\psi}_T:=\sum_{T'\in \widetilde{T}}\psi_{T'}$, we first note that for $f\in D^s_{p,\infty}$ and $\varphi\in \mathcal{S}(\bb{R}^d)$, 
\begin{equation*}
\abs{\scalarp{f,\varphi}}\leq \sum_{T\in \mathcal{T}}\norm{\psi_T(D)f\widetilde{\psi}_T(D)\varphi}_{L^1}\leq \sum_{T\in \mathcal{T}}\norm{\psi_T(D)f}_{L^\infty}\norm{\widetilde{\psi}_T(D)\varphi}_{L^1}.
\end{equation*}
Using \Lemref{Lem:TechnicalBelow} below (with $g=\mathcal{F}^{-1}\{\psi_T\hat{f}(T\xi)\}$) we thus get
\begin{align}\label{eq:appendixproof}
\abs{\scalarp{f,\varphi}}&\leq C_1 \sum_{T\in \mathcal{T}}\abs{T}^{1/p}\norm{\psi_T(D)f}_{L^p}\norm{\widetilde{\psi}_T(D)\varphi}_{L^1} \notag\\
&\leq C_1\norm{f}_{D^s_{p,\infty}}\sum_{T\in \mathcal{T}}\abs{T}^{1/p}\omega_T^{-s}\norm{\widetilde{\psi}_T(D)\varphi}_{L^1}\notag\\
&\leq C_2\norm{f}_{D^s_{p,\infty}}\norm{\left\{\norm{\widetilde{\psi}_T(D)\varphi}_{L^1}\right\}_{T\in \mathcal{T}}}_{\ell^1_{\omega^{\gamma/p-s}}},
\end{align}
since $|T|=|Q|^{-1}|Q_T|\leq |Q|^{-1}\omega_T^\gamma$ according to \Defref{Def:SAC}\eqref{Def:SAC5}. Applying \eqref{eq:Operatornorm} \Pageref{eq:Operatornorm} we may continue on \eqref{eq:appendixproof} and write
\begin{equation*}
\abs{\scalarp{f,\varphi}}\leq C_3\norm{f}_{D^s_{p,\infty}}\norm{\varphi}_{D^{\gamma/p-s}_{1,1}}\leq C_4\norm{f}_{D^s_{p,\infty}}p_N(\varphi),
\end{equation*}
for sufficiently large $N$ since $\mathcal{S}(\bb{R}^d)\hookrightarrow D^s_{1,1}$. We conclude that $D^s_{p,\infty}\hookrightarrow\mathcal{S}'(\bb{R}^d)$ which proves \Theenuref{THE:DECOMPOSITIONTHEOREM}{THE:DECOMPOSITIONTHEOREM1}.

The proof of \Theenuref{THE:DECOMPOSITIONTHEOREM}{THE:DECOMPOSITIONTHEOREM2} follows directly from \Theenuref{THE:DECOMPOSITIONTHEOREM}{THE:DECOMPOSITIONTHEOREM1} and the arguments in \cite[Page 150]{BorupNielsen2008}. 

To prove \Theenuref{THE:DECOMPOSITIONTHEOREM}{THE:DECOMPOSITIONTHEOREM3} we let $f\in D^s_{p,q}$ and choose $I\in C^\infty_c(\bb{R}^d)$ with $0\leq I(\xi)\leq 1$ and $I(\xi)\equiv 1$ in a neighbourhood of $\xi=0$. Also, we define $(\widetilde{f})^{\widehat{}}:=I\hat{f}$ and
\begin{equation*}
\widetilde{f}_\varepsilon:=\mathcal{F}^{-1}\left\{\varphi_\varepsilon\ast \left(\widetilde{f}~\right)^{\widehat{}}~\right\}\in \mathcal{S}(\bb{R}^d),
\end{equation*}
with $\varphi_\varepsilon(\xi):=\varepsilon^{-d}\varphi(\xi/\varepsilon)$ and $\varphi$ being a compactly supported mollifier. Since $\supp(I)$ is compact, we may choose a \emph{finite} subset $T^*\subset \mathcal{T}$, such that $\supp(I)\subset \cup_{T\in T^*}Q_T$ and $\sum_{T\in T^*}\psi_T(\xi)\equiv 1$ on $\supp(I)$. Using \Lemref{Lem:BOproof2} below we obtain
\begin{align*}
\|\widetilde{f}\|_{L^p}&=\norm{\mathcal{F}^{-1}I\mathcal{F}\left(\mathcal{F}^{-1}\left(\sum_{T\in T^*}\psi_T\cdot \hat{f}\right)\right)}_{L^p}\\
&\leq C\sum_{T\in T^*}\norm{\mathcal{F}^{-1}I}_{L^{\tilde{p}}}\norm{\psi_T(D)f}_{L^p}<\infty,
\end{align*}
with $\tilde{p}=\min\{1,p\}$. The dominated convergence theorem thus yields
\begin{equation*}
\norm{\widetilde{f}-\widetilde{f}_\varepsilon}_{D^s_{p,q}}\leq C\norm{\left\{\norm{\widetilde{f}-\widetilde{f}_\varepsilon}_{L^p}\right\}_{T\in \mathcal{T}}}_{\ell^q_{\omega^s}}\rightarrow 0,\quad \text{as }\varepsilon\rightarrow 0,
\end{equation*}
so the proof is done if we can show that $\|f-\widetilde{f}\|_{D^s_{p,q}}$ can be made arbitrary small by choosing $\widetilde{f}$ appropriately. To show this, we define the set $T_{\circ}:=\{T\in \mathcal{T}~|~I(\xi)\equiv 1 \text{ on }\supp(\psi_T)\}$. Denoting the complement $T_{\circ}^c$, \Lemref{Lem:BOproof2} below yields
\begin{align*}
\norm{f-\widetilde{f}}_{D^s_{p,q}}^q&=\sum_{T\in T_{\circ}^c}\omega_T^{sq}\norm{\mathcal{F}^{-1}\left(\psi_T\left(\hat{f}-I\hat{f}\right)\right)}_{L^p}^q\\
&\leq C_1\sum_{T\in T_{\circ}^c}\omega_T^{sq}\left(\norm{\psi_T(D)f}_{L^p}+\norm{\mathcal{F}^{-1}I\mathcal{F}\left(\psi_T(D)f\right)}_{L^p}\right)^q\\
&\leq C_2\sum_{T\in T_{\circ}^c}\omega_T^{sq}\norm{\psi_T(D)f}_{L^p}^q.
\end{align*}
Finally, since $f\in D^s_{p,q}$ we can choose $\supp(I)$ large enough, such that $\|f-\widetilde{f}\|_{D^s_{p,q}}<\varepsilon$, for any given $\varepsilon>0$. This proves \Theenuref{THE:DECOMPOSITIONTHEOREM}{THE:DECOMPOSITIONTHEOREM3}.
\end{proof}
In the proof of \Theref{THE:DECOMPOSITIONTHEOREM} we used the following two lemmas. A proof of \Lemref{Lem:TechnicalBelow} can be found in \cite[Lemma 3]{BorupNielsen2007} and a proof of \Lemref{Lem:BOproof2} can be found in \cite[Proposition 1.5.1]{Triebel2010}.
\begin{Lem}\label{Lem:TechnicalBelow}
Let $g\in L^p(\bb{R}^d)$ and $\supp(\hat{g})\subset\Gamma$, with $\Gamma\subset \bb{R}^d$ compact. Given an invertible affine transformation $T$, let $\hat{g}_T(\xi):=\hat{g}(T^{-1}\xi)$. Then for $0<p\leq q\leq \infty$,
\begin{equation*}
\norm{g_T}_{L_q}\leq C\abs{T}^{1/p-1/q}\norm{g_T}_{L_p},
\end{equation*}
for a constant $C$ independent of $T$.
\end{Lem}
\begin{Lem}\label{Lem:BOproof2}
Let $\Omega$ and $\Gamma$ be compact subsets of $\bb{R}^d$. Let $0<p\leq \infty$ and $\tilde{p}=\min\{1,p\}$. Then there exists a constant $C$ such that
\begin{equation*}
\norm{\mathcal{F}^{-1}M\mathcal{F}f}_{L^p}\leq C \norm{\mathcal{F}^{-1}M}_{L^{\tilde{p}}}\norm{f}_{L^p}
\end{equation*}
for all $f\in L^p(\bb{R}^d)$ with $\supp(\hat{f})\subset \Omega$ and all $\mathcal{F}^{-1}M\in L^{\tilde{p}}(\bb{R}^d)$ with $\supp(M)\subset \Gamma$.
\end{Lem}
\bibliographystyle{abbrv}

\end{document}